\documentclass{amsproc}
\usepackage{amssymb}

 \theoremstyle{plain}
\newtheorem{thm}{Theorem}[section]
 \theoremstyle{definition}
  \newtheorem{example}[thm]{Example}
  \theoremstyle{definition}
  \newtheorem{defn}[thm]{Definition}
  \theoremstyle{remark}
  \newtheorem{rem}[thm]{Remark}
  \theoremstyle{plain}
  \newtheorem{lem}[thm]{Lemma}

\begin{document}

\title{Groupoid Methods in Wavelet Analysis}

\author{Marius Ionescu}

\author{Paul S. Muhly}

\date{\today}

\address{6188 Kemeny Hall, Dartmouth College Hanover, NH 03755-3551 }
\curraddr{438 Malott Hall, Cornell University, Ithaca, NY, 14853}

\email{mionescu@math.cornell.edu}

\address{The University of Iowa, Iowa City, IA 52242-1419}

\email{pmuhly@math.uiowa.edu}

\dedicatory{Dedicated to the memory of George W. Mackey}

\thanks{This work was partially supported by the National Science Foundation,
DMS-0355443.}



\begin{abstract}
We describe how the Deaconu-Renault groupoid may be used in the study
of wavelets and fractals.
\end{abstract}

\maketitle

\section{Introduction}

This note serves two purposes. First, we want to describe investigations
that we are undertaking which are inspired in large part by work of
Palle Jorgensen and his collaborators, particularly Ola Bratteli,
Dorin Dutkay and Steen Pedersen%
\footnote{A number of relevant papers are cited in the references to this paper,
but for a more comprehensive list, the reader should consult the books
\cite{BJ02} and \cite{pJ06}.%
}. In their papers one finds a rich theory of wavelets on the one hand
and topics in fractal analysis on the other. Further, the analysis
in these papers is laced with representations of the Cuntz relations
- finite families of isometries $\{ S_{i}\}_{i=1}^{n}$ such that
$\sum_{i=1}^{n}S_{i}S_{i}^{*}=1$. Very roughly speaking, these authors
show that much of the analysis of wavelets and fractals that has appeared
in recent years may be illuminated in terms of special representations
of the Cuntz relations. Indeed, some of the most important advances
are made by choosing an appropriate representation for these relations.
Our motivation was to understand the extent to which the use of the
Cuntz relations is \emph{intrinsic} to the situation under consideration.
We wanted to separate intrinsically occurring representations of the
Cuntz relations from those that are imposed by special choices. We
hoped, thereby, to clarify the degrees of freedom that go into the
representations found in the work we are discussing.

As it turns out, the Cuntz isometries that arise in the work of Jorgensen
et. al. may be expressed in terms of representations of the Deaconu-Renault
groupoid associated to an appropriate local homemorphism of a compact
Hausdorff space. Our second purpose is to show how the $C^{*}$-algebra
of this groupoid is related to a number of other $C^{*}$-algebras
that one can attach to a local homeomorphism. In particular, we show
that the $C^{*}$-algebra may be realized as a Cuntz-Pimsner algebra
in two different ways and that, in general, it is a quotient of certain
other $C^{*}$-algebras that one may build from the local homeomorphism.

Recall that a wavelet is usually understood to be a vector $\psi$
in $L^{2}(\mathbb{R})$ such that the family \[
\{ D^{j}T^{k}\psi\;:\: j,k\in\mathbb{Z}\}\]
is an orthonormal basis for $L^{2}(\mathbb{R})$, where $T$ is the
operator of translation by $1$, i.e., $T\xi(x)=\xi(x-1),$ $\xi\in
L^{2}(\mathbb{R}),$
and where $D$ is dilation by $2$, i.e., $D\xi(x)=\sqrt{2}\xi(2x)$.
One of the principal problems in the study of wavelets is to construct
them with various pre-assigned properties. That is, one wants to {}``tune''
the parameters that enter into the analysis of wavelets so that the
wavelet one constructs exhibits the prescribed properties. So one's
first task is to identify those parameters and to understand the relations
among them.

Fractals, on the other hand, are spaces that possess some sort of
scaling. That is, as is customarily expressed, fractals exhibit the
same features at all scales. How to make this statement precise and
how to construct such spaces in useful ways are, of course, the objects
of considerable research. Most of the known examples of fractals are
closely connected to spaces endowed with a local homeomorphism that
is not injective. This may seem like a banal oversimplification, but
reflection on it does lead to natural representations of the Cuntz
relations, as we shall see, that are intrinsic to the geometry of
the situation. Since wavelets have a natural scaling built into them,
it is natural contemplate the possibility of building natural wavelet-like
orthonormal bases in $L^{2}$-spaces erected on fractals. This is
indeed possible, and much of the work by Jorgensen and his co-authors
has been devoted to realizing the possibilities.

Our contribution is to observe that the Deaconu-Renault groupoid associated
with a local homeomorphism of a compact Hausdorff space provides a
natural environment in which to set up fractal analysis, and that
the $C^{*}$-algebras of the groupoids carry natural, geometrically
induced families of isometries. The representations that Jorgensen
and his collaborators study come from representations of this groupoid.
Further, wavelets and other orthonormal bases on the fractals are
seen to be artifacts of the representation theory of the groupoid.
In short, groupoids help to clarify constructions of both fractals
and wavelets and help to analyze the parameters involved.

\section{The setup}

Throughout this note, $X$ will denote a fixed compact Hausdorff space
and $T:X\rightarrow X$ will denote a surjective local homeomorphism.
One can relax these hypotheses in various ways and in various situations,
but we shall not explore the possibilities here. The principal examples
to keep in mind are the following.

\begin{example}
Let $X$ be the circle or torus $\mathbb{T}$. The local homeomorphism
in this case is also an endomorphism of the abelian group structure
on $\mathbb{T}$: $T(z)=z^{N},$ $z\in\mathbb{T}$, where $N$ is
a natural number. The case when $N=2$, provides a link to {}``classical''
wavelets. 
\end{example}

\begin{example}
Let $A$ be an $n\times n$ dilation matrix. That is, suppose $A$
has integer entries and that the determinant of $A$ has absolute
value $d$ (which must be a positive integer) that is greater than
$1$. If we view the $n$-torus, $\mathbb{T}^{n}$, as the quotient
group $\mathbb{R}^{n}/\mathbb{Z}^{n}$, then $A$ induces a local
homeomorphism $T:\mathbb{T}^{n}\rightarrow\mathbb{T}^{n}$ via the
formula $T(x+\mathbb{Z}^{n})=Ax+\mathbb{Z}^{n}$,
$x+\mathbb{Z}^{n}\in\mathbb{T}^{n}$.
It is not difficult to see that $T$ is $d$-to-$1$. 
\end{example}
%\smallskip

\begin{example}
In this example we connect our discussion to the theory of iterated
function systems, which is one of the main ways to construct fractals
\cite{mB88}. Assume $X$ is a compact space endowed with a metric,
$d$ say, and let $(\varphi_{1},\dots,\varphi_{n})$ be a system of
maps on $X$ for which there are constants $c_{1}$ and $c_{2}$ such
that $0<c_{1}\leq c_{2}<1$ and such that $c_{1}d(x,y)\leq
d(\varphi_{i}(x),\varphi_{i}(y))\leq c_{2}d(x,y)$
for each $i$. Then each $\varphi_{i}$ is homeomorphism onto its
range. Also, the family $(\varphi_{1},\dots,\varphi_{n})$ induces
a map $\Phi$ on the space of non-empty closed (and hence compact)
subsets $K$ of $X$ via the formula \[
\Phi(K)=\cup_{i=1}^{n}\varphi_{i}(K).\]
It is then easy to see that $\Phi$ is a strict contraction in the
Hausdorff metric on the space of nonempty closed subsets of $X$ and
so there is a unique nonempty compact subset $K$ of $X$ such that
$\Phi(K)=K$. This $K$ is called \emph{the} invariant subset of the
system. It is the fractal associated with the system. We shall assume
that $X$ is the invariant set. It is important to note that there
may be overlap between $\varphi_{i}(X)$ and $\varphi_{j}(X)$ for
$i\ne j$. Consequently, the $\varphi_{i}$ need not be branches of
the inverse of a local homeomorphism. One way to {}``get around''
this limitation is to \emph{lift} the system in the sense of \cite[Page
155]{mB88}.
For this purpose, let $E^{\infty}$ the space of infinite words over
the alphabet $E=\{1,\dots,n\}$. Then in the product topology $E^{\infty}$
is compact and we can give $E^{\infty}$ a complete metric such that
the maps $\sigma_{i}:E^{\infty}\rightarrow E^{\infty}$ defined by
the formula $\sigma_{i}(w)=(i,w_{1},w_{2},\ldots)$, where
$w=(w_{1},w_{2},\ldots)$,
are contractions of the same type as the $\varphi_{i}$. The iterated
function system on $X\times E^{\infty}$ ,
$(\tilde{\varphi_{1}},\dots,\tilde{\varphi_{n}})$,
defined by the formula\[
\tilde{\varphi_{i}}(x,w)=(\varphi_{i}(x),\sigma_{i}(w))\]
 then has a unique nonempty closed invariant subset $\tilde{X}$ of
$X\times E^{\infty}$. That is
$\cup_{i=1}^{n}\tilde{\varphi_{i}}(\tilde{X})=\tilde{X}$.
The system $(\tilde{\varphi_{1}},\dots,\tilde{\varphi_{n}})$ on $\tilde{X}$
is called the \emph{lifted system.} The ranges of the $\tilde{\varphi_{i}}$
are disjoint and so there is a local homeomorphism $T$ of $\tilde{X}$
such that the $\tilde{\varphi_{i}}$ are the branches of the inverse
of $T$. As is discussed in Section 4.6 of \cite{mB88}, the
systems$(\varphi_{1},\dots,\varphi_{n})$
and $(\tilde{\varphi_{1}},\dots,\tilde{\varphi_{n}})$ share many
features in common and, from some points of view, are interchangeable. 
\end{example}

\section{The Deaconu-Renault Groupoid}

The Deaconu-Renault groupoid associated with the local homeomorphism
$T:X\to X$ is\[
G=\{(x,n,y)\in X\times\mathbb{Z}\times X\;:\; T^{k}(x)=T^{l}(y),n=k-l\}.\]
Two triples $(x_{1},n_{1},y_{1})$ and $(x_{2},n_{2},y_{2})$ are
composable if and only if $x_{2}=y_{1}$ and in that case,
$(x_{1},n_{1},y_{1})(x_{2},n_{2},y_{2})=(x_{1},n_{1}+n_{2},y_{2})$.
The inverse of $(x,n,y)$ is $(y,-n,x)$. A basis for the topology
on $G$ is given by the sets\[
Z(U,V,k,l):=\{(x,k-l,y)\in G:x\in U,y\in V\},\]
where $U$ and $V$ are open subsets of $X$ such that $T^{k}|U$,
$T^{l}|V$ are homeomorphisms and $T^{k}(U)=T^{l}(V)$. Thus $Z(U,V,k,l)$
is essentially the graph of $(T^{l}|_{V})^{-1}\circ(T^{k}|_{U})$
and is a $G$-set in the sense of Renault \cite{jR80}. The $G$-sets
form a pseudogroup $\mathfrak{G}$, viz. the pseudogroup of partial
homeomorphisms generated by $T$. The sets $Z(U,V,k,l)$ form a basis
for $\mathfrak{G}$. The groupoid $G$ is (isomorphic and homeomorphic
to) the groupoid of germs of $\mathfrak{G}$ precisely when the local
homeomorphism $T$ is essentially free, meaning that for no $m$ and
$n$ does $T^{m}=T^{n}$ on any open subset of $X$ \cite[Proposition 2.8]{jR00}.

The groupoid $G$ is $r$-discrete or \'{e}tale and so admits a Haar
system of counting measures. Consequently, we may define a $*$-algebra
structure on $C_{c}(G)$ as follows. For $f,g\in C_{c}(G)$ we
set\begin{eqnarray*}
f\ast g(x,k-l,y) & = & \sum f(x,m-n,z)\cdot g(z,(n+k)-(m+l),y),\end{eqnarray*}
 where the sum ranges over all $m$, $n$, and $z$ such that $T^{m}x=T^{n}z$,
and $T^{n+k}z=T^{m+l}y$, and we define \[
f^{*}(x,k-l,y)=\overline{f(y,l-k,x)}.\]
The algebra $C_{c}(G)$ can be completed to form a $C^{*}$-algebra,
denoted $C^{\ast}(G)$, in the norm\[
\Vert f\Vert:=\sup\Vert\pi(f)\Vert\]
where the supremum is taken over all $*$-homomorphisms of $C_{c}(G)$
into $B(H_{\pi})$ that are continuous with respect to the inductive
limit topology on $C_{c}(G)$ and the weak operator topology on $B(H_{\pi})$,
the algebra of operators on the Hilbert space of $\pi$, $H_{\pi}$.
We will discuss the representations of $C_{c}(G)$ more fully later,
but first we want to call attention to some special clopen relations
{}``on'' $X$. 

For fixed positive integers $m$ and $n$, we set $R_{n,m}:=\{(x,n-m,y)\in
G:T^{n}x=T^{m}y\}.$
Evidently, $R_{n,m}$ is a union of the basic sets $Z(U,V,m,n)$,
and so is open in $G$. It is also closed, since its complement is
open by virtue of being a union of sets of the form $Z(U,V,k,l)$,
with $(k,l)\ne(m,n)$. The sets $R_{n,m}$, with $m=n$, are of special
importance: $R_{0,0}$ may be identified with the diagonal $\Delta$
in $X\times X$, while for $k>0$, $R_{k,k}$ may be identified with
the relation $X\ast_{T^{k}}X:=\{(x,y):T^{k}(x)=T^{k}(x)\}$ in $X\times X$.
The $C^{*}$-algebra of $R_{k,k}$, $C^{*}(R_{k,k})$, which may be
identified with the closure of $C_{c}(R_{k,k})$ in $C^{*}(G)$, is
the cross sectional $C^{*}$-algebra of a matrix bundle over $X$
and, therefore, is a continuous trace $C^{*}$-algebra. (See \cite{aK85}
for a discussion of algebras of the form $C^{*}(R_{k,k})$.) The sequence
of inclusions $R_{0,0}\subset R_{1,1}\subset R_{2,2}\subset\cdots$
leads to the sequence of inclusions $C^{*}(R_{k,k})\subset C^{*}(R_{k+1,k+1})$,
$k=0,1,2,\ldots$, and, consequently, we see that if $R_{\infty}=\{(x,0,y)\;:\;
T^{n}x=T^{n}y\mbox{ for some }n\}=\bigcup R_{n,n}$
, then $ $$C^{\ast}(R_{\infty})$ is the inductive limit
$\underrightarrow{\lim}C^{\ast}(R_{n,n})$.
We note that $R_{\infty}$ is the kernel of the fundamental homomorphism
on $G$: $(x,n,y)\rightarrow n$, which implements the gauge automorphism
group $\{\gamma_{z}\}_{z\in\mathbb{T}}$ defined on $C_{c}(G)$ by
the formula $\gamma_{z}(f)(x,n,y)=z^{n}f(x,n,y)$. The algebra
$C^{*}(R_{\infty})$
is the fixed point algebra of $\{\gamma_{z}\}_{z\in\mathbb{T}}$,
also known as the \emph{core} of $C^{*}(G)$. For these things, and
more, we refer the reader to \cite{DKM01,jR00}.

It is a straightforward calculation, performed first by Deaconu \cite{vD95},
to see that the local homeomorphism $T$ on $X$ induces a $*$-endomorphism
$\alpha:C^{\ast}(R_{\infty})\to C^{\ast}(R_{\infty})$ defined by
the equation,\begin{equation}
\alpha(f)(x,0,y)=\frac{1}{\sqrt{\vert T^{-1}(Tx))\vert\vert
T^{-1}(Ty))\vert}}f(Tx,0,Ty),\label{eq:alphainfty}\end{equation}
 $f\in C_{c}(R_{\infty})$. Further, a similar calculation shows that
the function $S$ in $C_{c}(G)$ defined by the equation\begin{equation}
S(x,m-n,y)=\begin{cases}
\frac{1}{\sqrt{\vert T^{-1}(Tx)\vert}}, & \mbox{if }m=1,n=0,Tx=y,\\
0 & \mbox{otherwise,}\end{cases}\label{eq:isometry}\end{equation}
 is an isometry that implements $\alpha$ in the sense that\begin{equation}
\alpha(f)=SfS^{\ast},\label{eq:Stacey1}\end{equation}
 $f\in C_{c}(R_{\infty})$. In particular, observe that\begin{equation}
SS^{*}(x,k-l,y)=\frac{1}{\vert
T^{-1}(Tx)\vert}1_{R_{1,1}}(x,k-l,y).\label{eq:Stacey2}\end{equation}
 As we shall see, $S$ is the source of all the isometries in the
papers by Jorgensen et. al. It is an intrinsic feature of the $C^{*}$-algebra
that comes from the basic data: $X$ and the local homeomorphism $T$.
In fact, we have the following theorem, Theorem \ref{StaceyXpdRep},
that makes precise the assertion that $C^{\ast}(G)$ is the universal
$C^{\ast}$-algebra generated by $C^{*}(R_{\infty})$, $\alpha$,
and $S$. In fact, there are several different perspectives from which
to see how $C^{*}(G)$ is constructed from the space $X$ and local
homeomorphism $T$. We want to examine these and to compare them with
various approaches in the literature. Therefore the proof will be
given after further discussion. 

\begin{thm}
\label{StaceyXpdRep}Let $\tilde{\pi}:C^{\ast}(G)\to\mathcal{B}(H)$
be a $C^{*}$-representation. Define $\pi:C^{\ast}(R_{\infty})\to\mathcal{B}(H)$
by $\pi=\tilde{\pi}|_{C^{\ast}(R_{\infty})}$, and let $S_{+}=\tilde{\pi}(S)$.
Then 
\begin{enumerate}
\item $\pi(\alpha(f))=S_{+}\pi(f)S_{+}^{*}$; and 
\item $\pi(\mathcal{L}(f))=S_{+}^{\ast}\pi(f)S_{+}$, where
$\mathcal{L}(f)=S^{\ast}fS$
is the transfer operator associated with $\alpha$, \begin{equation}
\mathcal{L}(f)(x,0,y)=\frac{1}{\sqrt{\vert T^{-1}(x)\vert\vert
T^{-1}(y)\vert}}\sum_{\stackrel{Tu=x}{Tv=y}}f(u,0,v).\label{eq:transferinfty}
\end{equation}

\end{enumerate}
Conversely, given $(\pi,S_{+})$, where
$\pi:C^{\ast}(R_{\infty})\to\mathcal{B}(H)$
is a $C^{*}$-representation and $S_{+}$ is an isometry on $H$ such
that 1. and 2. are satisfied, then there is a unique representation
$\tilde{\pi}:C^{\ast}(G)\to\mathcal{B}(H)$ such that $\tilde{\pi}(f)=\pi(f)$
for all $f\in C^{\ast}(R_{\infty})$ and $\tilde{\pi}(S)=S_{+}$.
\end{thm}
Recall, next, that if $A$ is a $C^{*}$-algebra, then a
$C^{*}$-\emph{correspondence
over} $A$ is an $A$-$A$-bimodule $E$ such that $E_{A}$ is a Hilbert
$C^{*}$-module and the left action is given by a $C^{*}$-homomorphism
$\phi$ from $A$ into the bounded adjointable operators on $E$ \cite{MS98},
$\mathcal{L}(E)$. We write $\mathcal{K}(E)$ for the space of compact
operators on $E$, i.e., $\mathcal{K}(E)$ is the closed linear span
of the operators $\xi\otimes\eta^{*}$, $\xi,\eta\in E$, defined
by the formula $\xi\otimes\eta^{*}(\zeta):=\xi\langle\eta,\zeta\rangle$,
and we write $J$ for the ideal $\phi^{-1}(\mathcal{K}(E))$ in $A$.
A \emph{Cuntz-Pimsner covariant representation} of $E$ in a $C^{*}$-algebra
$B$ is a pair $(\pi,\psi)$, where $\pi$ is a $C^{*}$-representation
of $A$ in $B$ and $\psi$ is a map from $E$ into $B$ such that 

\begin{enumerate}
\item $\psi(\phi(a)\xi b)=\pi(a)\psi(\xi)\pi(b)$, for all $a,b\in A$ and
all $\xi\in E.$;
\item for all $\xi,\eta\in E$,
$\psi(\xi)^{*}\psi(\eta)=\pi(\langle\xi,\eta\rangle)$;
and
\item for all $a\in J$, $(\psi,\pi)^{(1)}(\phi(a))=\pi(a)$, where
$(\psi,\pi)^{(1)}$
is the representation of $\mathcal{K}(E)$ in $B$ defined by the
formula $(\psi,\pi)^{(1)}(\xi\otimes\eta^{*})=\psi(\xi)\psi(\eta)^{*}$,
$\xi\otimes\eta^{*}\in\mathcal{K}(E)$. 
\end{enumerate}
There is a $C^{*}$-algebra $\mathcal{O}(E)$ and Cuntz-Pimsner representation
$(k_{A},k_{E})$ of $E$ in $\mathcal{O}(E)$ that is universal for
all Cuntz-Pimsner representations of $E$. That is, if $(\pi,\psi)$
is a Cuntz-Pimsner representation of $E$ in a $C^{*}$-algebra $B$,
then there is a unique $C^{*}$-representation $\rho$ of $\mathcal{O}(E)$
in $B$ such that $\rho\circ k_{A}=\pi$ and $\rho\circ k_{E}=\psi$
. The representation $\rho$ is often denoted $\pi\times\psi$. This
was proved essentially by Pimsner in \cite{mP97} and in the form
stated here in \cite[Proposition 1.3]{FMR03}. 

\begin{defn}
The \emph{Deaconu $C^{*}$-correspondence} $\mathcal{X}$ over the
$C^{\ast}$-algebra $C(X)$ is the completion of $C(X)$ under the
inner product\[
\langle\xi,\eta\rangle(x)=\frac{1}{\vert
T^{-1}(x)\vert}\sum_{Ty=x}\overline{\xi(y)}\eta(y),\]
 with the left and right actions of $C(X)$ given by $(a\cdot\xi\cdot
b)(x)=a(x)\xi(x)b(Tx)$. 

The definition we have given is slightly different from the one given
in \cite{vD99}. He does not divide by $\vert T^{-1}(x)\vert$. However,
it is easy to see that the two $C^{*}$-correspondences are isomorphic.
The following theorem is due to Deaconu \cite[Propositions 3.1 and 3.3]{vD99}.
The formulation we present is that of \cite[Theorem 7]{DKM01}, which
is slightly more general. The proof in \cite{DKM01} is based on the
gauge invariant uniqueness theorem found in \cite[Theorem 4.1]{FMR03}.
\end{defn}
\begin{thm}
\label{DeaconuIsoThm}(Deaconu) Define $\iota:C(X)\to C^{\ast}(G)$,
by the equation \[
\iota(\varphi)(x,k-l,y)=\varphi(x)1_{R_{0,0}}(x,k-l,y),\]
 and $\psi:\mathcal{X}\to C^{*}(G)$, by the equation $\psi(\xi)=\iota(\xi)S$.
Then $(i,\psi)$ is a faithful Cuntz-Pimsner covariant representation
of $(C(X),\mathcal{X})$ in $C^{\ast}(G)$, whose image generates
$C^{*}(G)$ and gives an isomorphism between $C^{*}(G)$ and
$\mathcal{O}(\mathcal{X})$. 
\end{thm}
In \cite{rE03}, Exel introduced a crossed product associated to an
endomorphism $\alpha$ of a $C^{*}$-algebra $A$ and transfer operator
$\mathcal{L}$ for $\alpha$. That is, $\mathcal{L}$ is a positive
operator on $A$ that satisfies the equation
$\mathcal{L}(a\alpha(b))=\mathcal{L}(a)b$
for all $a,b\in A$. Exel's crossed product, denoted
$A\rtimes_{\alpha,\mathcal{L}}\mathbb{N}$
$ $can also be described as a relative Cuntz-Pimsner algebra, as
was accomplished by Brownlowe and Raeburn in \cite{BR06}. We adopt
their perspective and assume also that $A$ is unital, but we don't
assume that $\alpha$ is unital. Let $M_{\mathcal{L}}$ denote the
completion of $A$ in the inner product $\langle
a,b\rangle:=\mathcal{L}(a^{*}b)$,
and give $M_{\mathcal{L}}$ the right and left actions of $A$ defined
by the formulae $m\cdot a:=m\alpha(a)$ and $a\cdot m=\phi(a)m=am$.
As a left $A$-module $M_{\mathcal{L}}$ is cyclic and the image of
$1$ in $M_{\mathcal{L}}$ is a cyclic vector, which we denote by
$\xi_{0}$. If $(\pi,\psi)$ is a Cuntz-Pimsner representation of
$M_{\mathcal{L}}$ in a $C^{*}$-algebra $B$, then the image of $\xi_{0}$
in $B$, $\psi(\xi_{0})$, is an isometry $V$, say. Then $(\pi,\psi)$
is completely determined by $\pi$ and $V$ in the following sense:
Let $\pi$ be a representation of $A$ in a $C^{*}$-algebra $B$,
let $V$ be an isometry in $B$, and define $\psi:M_{\mathcal{L}}\to B$
by the formula, $\psi(\phi(a)\xi_{0})=\pi(a)V$, then $(\pi,\psi)$
is a Cuntz-Pimsner representations of $M_{\mathcal{L}}$ in $B$ if
and only if the following equations CP1., CP2. and CP3. are satisfied:

\begin{enumerate}
\item [CP1.]$V\pi(a)=\pi(\alpha(a))V$ for all $a\in A$; 
\item [CP2.]$V^{*}\pi(a)V=\pi(\mathcal{L}(a))$ for all $a\in A$; and
\item [CP3.]$\pi(a)=(\psi,\pi)^{(1)}(\phi(a))$, for all $a\in J$.
\end{enumerate}
\begin{thm}
\label{thm:ExelCP}In the context of our groupoid, $G$, let
$A=C^{*}(R_{\infty})$,
let $\alpha$ be the endomorphism of $A$ defined by equation
(\ref{eq:alphainfty}),
let $\mathcal{L}$ be the associated transfer operator (\ref{eq:transferinfty})
and let $M_{\mathcal{L}}$ be the correspondence over $A$ defined
by Brownlowe and Raeburn that we just described. Then the identity
representation $\iota$ mapping $C^{*}(R_{\infty})$ into $C^{*}(G)$
together with the isometry $S$ defined by equation (\ref{eq:isometry}),
determine a Cuntz-Pimsner representation $(\iota,\psi)$ of $M_{\mathcal{L}}$
in $C^{*}(G)$ that implements an isomorphism of $\mathcal{O}(M_{\mathcal{L}})$
onto $C^{*}(G)$.
\end{thm}
\begin{proof}
Equation CP.1 follows from equation (\ref{eq:Stacey1}) and equation
CP.2, which is the same as the second equation of Theorem \ref{StaceyXpdRep},
is a straightforward calculation. We need to verify equation CP.3.
Since $\xi_{0}$ is a cyclic vector for the left action of $A$ on
$M_{\mathcal{L}}$, $\mathcal{K}(M_{\mathcal{L}})$ is the closed
linear span of elements of the form $\phi(a)\xi_{0}\otimes\xi_{0}^{*}\phi(b)$,
where $a$ and $b$ range over $A$. So, if $\phi(a)$ is compact,
there is a sequence whose terms are of the form
$\sum_{i}\phi(a_{i})\xi_{0}\otimes\xi_{0}^{*}\phi(b_{i})$
that converges to $\phi(a)$ in $\mathcal{K}(M_{\mathcal{L}})$. So,
if we apply $\iota(a)$ to an element of the form
$\iota(b)S=\psi(\phi(b)\xi_{0})$,
then we may write the following equation \begin{eqnarray}
\iota(a)\iota(b)S & = &
\psi(\phi(a)\phi(b)\xi_{0})=\lim\psi(\sum_{i}\phi(a_{i})\xi_{0}\otimes\xi_{0}^{*
}\phi(b_{i})(\phi(b)\xi_{0}))\label{eq:defining}\\
 & = & \lim\sum_{i}\iota(a_{i})SS^{*}\iota(b_{i})\iota(b)S\nonumber \\
 & = & \lim\sum_{i}\iota(a_{i})S\iota(\mathcal{L}(b_{i}b))\nonumber \\
 & = & \lim\sum_{i}\iota(a_{i})\iota(\alpha\circ\mathcal{L}(b_{i}b))S\nonumber
\\
 & = &
\lim\sum_{i}(\psi,\iota)^{(1)}(\phi(a_{i})\xi_{0}\otimes\xi_{0}^{*}\phi(b_{i}
))(\iota(b)S)\nonumber \\
 & = & (\psi,\iota)^{(1)}(\phi(a))(\iota(b)S).\nonumber \end{eqnarray}
By \cite[Lemma 4.4.1]{FMR03} \begin{equation}
\psi(\xi\otimes\eta^{*}(\phi(b)\xi_{0}))=(\psi,\iota)^{(1)}(\xi\otimes\eta^{*}
)(\iota(b)S),\label{eq:definingbis}\end{equation}
which shows that for all $T\in\mathcal{K}(M_{\mathcal{L}}),$
$(\psi,\iota)^{(1)}(T)$
is determined by its values on elements of the form $\iota(b)S$.
Thus, equations (\ref{eq:defining}) and (\ref{eq:definingbis}) together
show that if $a\in J$, then $\iota(a)=(\psi,\iota)^{(1)}(\phi(a))$.
Thus $(\iota,\psi)$ is a Cuntz-Pimsner representation, the range
of which clearly generates $C^{*}(G)$. So all we need to show is
that $\iota\times\psi$ is injective. But this is immediate from the
injectivity of $\iota$, by the gauge-invariant uniqueness theorem
\cite[Theorem 4.1]{FMR03}. 
\end{proof}
\emph{Proof of Theorem \ref{StaceyXpdRep}.} The fact that conditions
1. and 2. of the theorem are satisfied is an easy calculation. The
{}``converse'' assertion follows from Theorem \ref{thm:ExelCP}
because, as is easily seen, if $(\pi,S_{+})$ are given, acting on
a Hilbert space $H$, say, then we obtain a Cuntz-Pimsner representation
$(\pi,\psi)$ of $M_{\mathcal{L}}$ by setting $\psi(\xi):=\pi(\xi)S_{+}$.
This representation ``integrates'' to give a $C^{*}$-representation
of $\mathcal{O}(\mathcal{X})$, which by Theorem \ref{DeaconuIsoThm}
is $C^{*}(G)$.  $\square$

\section{Filter Banks}

\begin{defn}
A family $\{ m_{i}\}_{i=1,\dots,N}\subseteq\mathcal{X}$ is called
a \emph{filter bank} if it is an orthonormal basis for $\mathcal{X}$. 
\end{defn}
This means that $\langle m_{i},m_{j}\rangle=0$ if $i\neq j$, and
$\langle m_{i},m_{i}\rangle=1$. Note that this last condition is
much stronger than asserting that each $m_{i}$ has norm $1$. In
general a module $\mathcal{X}$ need not have an orthonormal basis.
Even some modules built on $\mathbb{T}^{n}$ with the map $z\to Az$
may fail to have orthonormal bases. However, on $\mathbb{T}^{1}$
they exist.

\begin{defn}
If $\{ m_{i}\}_{i=1,\dots,N}$ is a filter bank, we call $m_{1}$
the \emph{low pass filter} and the rest \emph{high pass filters}. 
\end{defn}
One problem of great importance is to decide when a function $m$
in $\mathcal{X}$ satisfying $\langle m,m\rangle=1$ can be completed
to an orthonormal basis, i.e., when can such a function $m$ be viewed
as a low pass filter in a filter bank. This depends to a great extent
upon the underlying geometry of the situation under consideration,
as Packer and Rieffel have shown \cite{PR03,PR04}.

We note, too, that while we have been emphasizing the topological
situation, there is a Borel version of our analysis. In this situation
Borel orthonormal bases always exist and low pass filters can be completed
to a filter bank.

\begin{thm}
\label{StaceyXpd2}Define $\beta:C(X)\to C(X)$ by $\beta(f)=f\circ T$,
$f\in C(X)$, and adopt the notation of Theorem \ref{DeaconuIsoThm}.
The following assertions are valid in $C^{*}(G)$ : 
\begin{enumerate}
\item $\iota(\beta(a))S=S\iota(a)$, for $a\in C(X).$ 
\item If $\{ m_{1},\dots,m_{n}\}$ is a filter bank and if $S_{i}:=\psi(m_{i})$,
then $\{ S_{i}\}$ is a Cuntz family of isometries in $C^{*}(G)$
such that\[
\iota(\beta(a))S_{i}=S_{i}\iota(a).\]

\item For all $a\in C(X)$\begin{equation}
\iota(\beta(a))=\sum_{i=1}^{n}S_{i}\iota(a)S_{i}^{*}.\label{eq:inner}
\end{equation}

\end{enumerate}
\end{thm}
The proof of Theorem \ref{StaceyXpd2} is a straightforward calculation
and so will be omitted. Nevertheless, there are several useful points
to be raised about the result.

Suppose, quite generally, that $A$ is a $C^{*}$-algebra and that
$\alpha$ is an endomorphism of $A$. Then the powers of $\alpha$
can be used to build an inductive system $(\{
A_{n}\}_{n=0}^{\infty},\{\alpha_{m,n}\}_{m\geq n})$
in a familiar fashion: one takes $A_{n}$ to be $A$ for every $n$
and sets $\alpha_{m,n}:=\alpha^{m-n}$, when $m\geq n$. The inductive
limit of this system, $A_{\infty}$, exists, but may be zero. In the
event the limit is not $0$, then, as Stacey proves in Proposition
3.2 of \cite{pS93}, there is, for each positive integer $n$, a unique
$C^{*}$-algebra $B$ and a pair $(\iota,\{ t_{i}\}_{i=1}^{n})$ consisting
of a $*$-homomorphism $\iota:A\to B$ such that
$\overline{\iota}(1_{M(A)})=1_{M(B)}$,
where $\overline{\iota}$ denotes the extension of $\iota$ to the
multiplier algebra of $A$, $M(A)$, and a family of isometries in
the multiplier algebra of $B$, $\{ t_{i}\}_{i=1}^{n}\subseteq M(B)$,
such that

\begin{enumerate}
\item $\{ t_{i}\}_{i=1}^{n}$ is a Cuntz family of isometries, i.e.,
$t_{i}^{*}t_{j}=\delta_{ij}1_{M(B)}$,
for $i,j=1,2,\ldots,n$, and $\sum_{i=1}^{n}t_{i}t_{i}^{*}=1_{M(B)}$.
When $n=1$, $t=t_{1}$ is simply an isometry.
\item For all $a\in A$, $\iota(\alpha(a))=\sum_{i=1}^{n}t_{i}\iota(a)t_{i}^{*}$.
\item If $(\pi,\{ T_{i}\}_{i=1}^{n})$ is a family consisting of a
$C^{*}$-representation
of $A$ on a Hilbert space $H$ and a Cuntz family of isometries $\{
T_{i}\}_{i=1}^{n}$
in $B(H)$, then there is a nondegenerate representation $(\pi\times T)$
of $B$ on $H$ so that $(\pi\times T)\circ\iota=\pi$ and $(\pi\times
T)(t_{i})=T_{i}$,
$i=1,2\ldots,n$. (The family $(\pi,\{ T_{i}\}_{i=1}^{n})$ is called
a \emph{Cuntz-covariant representation of order} $n$ of the system
$(A,\alpha)$.) 
\item $B$ is the $C^{*}$-algebra generated by $\iota(A)$ and elements
of the form $\iota(a)t_{i}$, $i=1,2,\ldots,n$ and $a\in A$. 
\end{enumerate}
\begin{defn}
The $C^{*}$-algebra $B$ just described is called \emph{the Stacey
crossed product of order $n$ determined by $A$ and $\alpha$}, and
is denoted $A\rtimes_{n}^{\alpha}\mathbb{N}$. \end{defn}

Note that when $n=1$, the endomorphism in a Stacey crossed product
of order $1$ cannot be unital if the embedding $\iota$ is injective.
This happens if and only if there is a Cuntz-covariant representation
$(\pi,T)$ of order $1$ with a faithful $\pi$.$ $ In the setting
of Theorem \ref{StaceyXpdRep}, it is clear that $\alpha$ is not
unital by virtue of equation (\ref{eq:Stacey1}). Also, by virtue
of equation 1. in the statement of that theorem it is natural to speculate
about the relation between $C^{*}(G)$ and the Stacey crossed product
of order $1$ determined by $C^{*}(R_{\infty})$ and $\alpha$. It
turns out that the crossed product that Exel would associate to
$C^{*}(R_{\infty})$,
$\alpha$, and $\mathcal{L}$, in \cite{rE03} and which he would
denote by $C^{*}(R_{\infty})\rtimes_{\alpha,\mathcal{L}}\mathbb{N}$,
is isomorphic to $C^{*}(R_{\infty})\rtimes_{1}^{\alpha}\mathbb{N}$
by his \cite[Theorem 4.7]{rE03}. On the other hand, Brownlowe and
Raeburn show that Exel's algebra
$C^{*}(R_{\infty})\rtimes_{\alpha,\mathcal{L}}\mathbb{N}$
is isomorphic to the relative Cuntz-Pimsner algebra determined the
ideal $\overline{A\alpha(A)A}\cap J$, where $A=C^{*}(R_{\infty})$%
\footnote{If $E$ is a $C^{*}$-correspondence over a $C^{*}$-algebra $A$
and if $K$ is an ideal in $J$, then the \emph{relative Cuntz-Pimsner
algebra} determined by $K$, $\mathcal{O}(K;E)$, is the universal
$C^{*}$-algebra for representations of $E$, $(\pi,\psi)$, that
have all the properties of a Cuntz-Pimsner representation except that
the equation $\pi(a)=(\psi,\pi)^{(1)}(\phi(a))$ is assumed to hold
only for $a\in K$. See \cite{MS98} and \cite{FMR03}, where the
basic theory of such algebras is developed.%
}. Now in this situation $J$ coincides with $A$ because $\phi(1)=\phi(P)$,
where $P=SS^{*}$, and because $\phi(P)=\xi_{0}\otimes\xi_{0}^{*}$.
On the other hand, the ideal $\overline{A\alpha(A)A}$ is proper.
Thus, the relative Cuntz-Pimsner algebra determined by
$\overline{A\alpha(A)A}\cap J$
has the Cuntz-Pimsner algebra $\mathcal{O}(M_{\mathcal{L}})$ as a
proper quotient, by \cite[Proposition 3.14]{FMR03}. So, in our setting,
we see that $C^{*}(G)$ is a proper quotient of
$C^{*}(R_{\infty})\rtimes_{\alpha,\mathcal{L}}\mathbb{N}\simeq
C^{*}(R_{\infty})\rtimes_{1}^{\alpha}\mathbb{N}$.
On the other hand, Theorem \ref{StaceyXpd2} suggests that $C^{*}(G)$
may be the Stacey crossed product $C(X)\rtimes_{n}^{\beta}\mathbb{N}$,
but we are unable to determine the precise circumstances under which
this may happen. Nevertheless, as Theorem \ref{StaceyXpd2} shows,
$C^{*}(G)$ contains a Cuntz covariant representation of order $n$
of $(C(X),\beta)$, and therefore any $C^{*}$-representation of $C^{*}(G)$
produces automatically a Cuntz-covariant representation of $(C(X),\beta)$.
These are the starting point of Bratteli and Jorgensen's analysis
\cite[Proposition 1.1]{BJ97}.

\section{Representations of $C^{*}(G)$}

Renault worked out the structure theory of the most general representation
of any groupoid $C^{*}$-algebra in \cite{jR87}. We discuss here
certain aspects of it in our special setting that is relevant for
applications to wavelets. Let $\pi:C^{*}(G)\to\mathcal{B}(\mathcal{H})$
be a $C^{*}$-representation, where $G$ for the moment is an arbitrary
locally compact groupoid with Haar system $\{\lambda^{u}\}_{u\in G^{(0)}}$.
Then $\pi$ determines and is determined by a triple $(\mu,\mathcal{H},U)$,
where $\mu$ is a quasi-invariant measure on $G^{(0)}=X$; $\mathcal{H}$
is a (Borel) Hilbert bundle on $X$, and $U$ is a representation
of $G$ on $\mathcal{H}$. The relation between $\pi$ and the triple
$(\mu,\mathcal{H},U)$ is expressed through the equation \[
\pi(f)\xi(u)=\int_{G^{u}}f(\gamma)(U(\gamma)\xi(s(\gamma)))\Delta^{\frac{1}{2}}
(\gamma)\, d\lambda^{u}(\gamma),\]
where $\xi$ is an $L^{2}(\mu)$-section of the bundle $\mathcal{H}$
and $\Delta$ is the modular function of the measure $\mu.$ In more
detail, let $\nu=\int_{G^{(0)}}\lambda^{u}\, d\mu(u)$ and let $\nu^{-1}$
be the image of $\nu$ under inversion. Then to say $\mu$ is quasi-invariant
is to say that $\nu$ and $\nu^{-1}$ are mutually absolutely continuous.
In this case, $\Delta$ is defined to be $\frac{d\nu^{-1}}{d\nu}$. 

Specializing now to the setting where our groupoid $G$ is the Deaconu-Renault
groupoid associated to the local homeomorphism $T$ on the compact
Hausdorff space $X$, it is not difficult to see that the measure
$\mu$ is quasi-invariant in the fashion just described if and only
if $\mu\circ T^{-1}\ll\mu$. In this event, if we let $D$ denote
the Radon-Nikodym derivative $\frac{d\mu\circ T^{-1}}{d\mu}$, then
the modular function $\Delta$ is given by the equation \[
\Delta(x,m-n,y)=\frac{D(x)D(Tx)\cdots D(T^{m-1}x)}{D(y)D(Ty)\cdots
D(T^{n-1}y)}.\]

A measurable function $D$ defines also a transfer operator
$\mathcal{L}_{D}^{*}:M(X)\to M(X)$
by the equation\begin{equation}
\mathcal{L}_{D}^{*}(\mu)(f):=\int_{X}\sum_{Ty=x}D(y)f(y)\,
d\mu(x).\label{eq:Transfer}\end{equation}
The relevance of the transfer operator $\mathcal{L}_{D}^{*}$ to our
situation was established by Renault in \cite[Theorem 7.1]{jR05}
and \cite[Proposition 4.2]{jR03}. We state a slightly modified version
of his results.

\begin{thm}
(Renault) Let $\mu$ be a probability measure on $X$. Then $\mu$
is quasi-invariant with respect to $G$ and admits $\Delta$ as Radon-Nikodym
derivative if and only if $\mathcal{L}_{D}^{*}(\mu)=\mu$. 
\end{thm}
In applications to wavelets, i.e. to the settings where $X=\mathbb{T}$
or $X=\mathbb{T}^{n}$ and $T$ is the power function $z\to z^{N}$
or $x+\mathbb{Z}^{n}\to Ax+\mathbb{Z}^{n}$ the measure that one usually
chooses is Lebesgue measure. Also, the bundle one chooses is the trivial
line bundle $\mathcal{H}=\mathbb{T}\times\mathbb{C}$ or
$\mathcal{H}=\mathbb{T}^{n}\times\mathbb{C}$
and the representation is the translation representation: $U(\gamma):\{
s(\gamma)\}\times\mathbb{C}\to\{ r(\gamma)\}\times\mathbb{C}$,\[
U(\gamma)(s(\gamma),c)=(r(\gamma),c),\]
$\gamma\in G$. But we note that some of the recent work of Dutkay
and Roysland \cite{DR06,DR07} can be formulated in the setting we
are describing by taking more complicated bundles and representations.

\section{An Example: Classical Wavelets}

We discuss how the constructs we have described can enter into analysis
of classical wavelets. In this setting, as we have indicated, $X$
is the circle or $1$-torus $\mathbb{T}$, the local homeomorphism
$T$ is given by squaring: $Tz=z^{2}$, the quasi-invariant measure
$\mu$ is Lebesgue measure on $\mathbb{T}$, the bundle $\mathcal{H}$
is the trivial one-dimensional bundle, and the representation $U$
is translation. The $L^{2}$-sections of $\mathcal{H}$ is just $L^{2}(\mu)$
and if $\pi$ is the integrated form of the representation associated
to this data, then $\pi$ represents $C(X)$ (viewed as $\iota(C(X))$
in $C^{*}(G)$) as multiplication operators on $L^{2}(\mu)$. Further,
if ${\{ m}_{i}\}_{i=1,2}$ is a filter bank and if $S_{1}$ and $S_{2}$
are the isometries it determines as in Theorem \ref{StaceyXpd2},
then $\pi(S_{i})\xi(z)=m_{i}(z)\xi(z^{2})$, $i=1,2$. Thus
$(\pi(S_{1}),\pi(S_{2}))$
is a Cuntz family on $L^{2}(\mu)$ such that\begin{eqnarray*}
\pi(\iota(\alpha(f))) & = & \pi(S_{1})\pi(\iota(f))\pi(S_{1})^{*}\\
 & + & \pi(S_{2})\pi(\iota(f))\pi(S_{2})^{*},\end{eqnarray*}
for all $f\in C(\mathbb{T})$, which is equation (\ref{eq:inner}).
We note in passing that there are many filter banks and that given
any $m_{1}$ in $\mathcal{X}$ such that $\langle m_{1},m_{1}\rangle=1$,
we obtain a filter bank $\{ m_{1},m_{2}\}$ if we take $m_{2}$ to
be the function defined by the equation \[
m_{2}(z):=z\overline{m_{1}(-z)}.\]
All other possibilities for $m_{2}$ are obtained from this choice
by multiplying it by $\theta(z^{2})$, where $\theta$ is a continuous
function of modulus $1$.

The key to building wavelets from the Cuntz relations is to build
the minimal unitary extension of $\pi(S_{1})$. This was observed
by Bratteli and Jorgensen in \cite{BJ97}. However, we follow Larsen
and Raeburn \cite{LR06} who use the inductive limit approach advanced
by Douglas \cite{rD69}. Here is the basic setup: Form the inductive
system\[
H_{n}\stackrel{S_{m,n}}{\longrightarrow}H_{m}\]
 where $H_{n}=L^{2}(\mathbb{T})$ for every $n$ and the ``linking
maps'' $S_{m,n}:H_{n}\to H_{m}$, are simply the powers of $\pi(S_{1})$:
$S_{m,n}=\pi(S_{1})^{m-n}$. We let $H_{\infty}$ denote the inductive
limit $\underrightarrow{\lim}(\{ H_{n}\},\{ S_{m,n}\}\}$, and we
let $S_{\infty,n}:H_{n}\to H_{\infty}$ denote the limit embeddings.
Then there is a unique unitary $U$ on $H_{\infty}$ so that\[
US_{\infty,n+1}=S_{\infty,n+1}\pi(S_{1})=S_{\infty,n},\]
for all $n$. This map $U$ is the minimal isometric extension of
$\pi(S_{1})$. We want to uncover a bit more structure in $U$. 

To this end, observe that for $m,n\ge0$,\begin{eqnarray*}
{(S}_{m,n}\xi)(z) & = & 2^{(m-n)/2}m_{1}(z)m_{1}(z^{2})\\
 & \cdots & m_{1}(z^{2^{(m-n)-1}})\xi(z^{2^{(m-n)}}).\end{eqnarray*}
Analysis of Dutkay and Jorgensen in \cite{DJ05} and \cite[Proposition 2.2]{DJ06}
leads to an explicit identification of $H_{\infty}$ with
$L^{2}(\mathbb{T}_{\infty},\tilde{\mu})$,
where $\mathbb{T}_{\infty}$ is the 2-adic solenoid and $\tilde{\mu}$
is a measure built from Lebesgue measure on $\mathbb{T}$ and the
transfer operator associated with $\vert m_{1}\vert^{2}$:\[
\mathcal{L}_{m_{1}}(f)(z)=\frac{1}{\vert T^{-1}(x)\vert}\sum_{w^{2}=z}\vert
m_{1}(w)\vert^{2}f(w).\]
 The representation $\pi$ of $C(\mathbb{T})$ on $L^{2}(\mu)$ extends
to a representation $\rho$ of $C(\mathbb{T})$ on
$L^{2}(\mathbb{T}_{\infty},\tilde{\mu})$
via the formula \[
\rho(f)\xi(z_{1},z_{2},\ldots)=f(z_{1})\xi(z_{1},z_{2},\ldots),\]
where $\xi\in L^{2}(\mathbb{T}_{\infty},\tilde{\mu})$, and where,
recall, points in $\mathbb{T}_{\infty}$ are sequences $(z_{1},z_{2},\ldots)$
such that $z_{k}=z_{k+1}^{2}$, for all $k\geq1$. Also, the measure
$\tilde{\mu}$ is quasi-invariant for the extension $T_{\infty}$
of the map $T$ on $\mathbb{T}$, defined on $\mathbb{T}_{\infty}$
by the formula $T_{\infty}(z_{1},z_{2},\ldots)=(z_{1}^{2},z_{1},z_{2},\ldots)$,
and $U$ is given by the formula
$$U\xi(z_{1},z_{2},\ldots)=\xi(z_{1}^{2},z_{1},z_{2},\ldots)J^{\frac{1}{2}}(z_{1}
,z_{2},\ldots),$$
where $J$ is the Radon-Nikodym derivative with respect to $\tilde{\mu}$
of the translate of $\tilde{\mu}$ by $T_{\infty}$ . The pair $(\rho,U)$
is a covariant pair: \[
\rho(f\circ T_{\infty})=U\rho(f)U^{-1},\]
for all $f\in C(\mathbb{T})$. To build the wavelet associated with
the filter bank, we need to get from $L^{2}(\mathbb{T}_{\infty},\tilde{\mu})$
to $L^{2}(\mathbb{R})$ in a unitary fashion that transforms $U$
into $D$, which recall is given by the formula $D\xi(x)=\sqrt{2}\xi(2x),$
and transforms $\rho$ into the representation $\tilde{\rho}(f)\xi(x)=f(e^{2\pi
ix})\xi(x)$.
This is accomplished with the aid of a famous theorem of Mallat.

\begin{thm}
\label{thm:Mallat}(Mallat \cite[Theorem
2]{sM89}) Suppose $m_{1}$, which is a unit vector
in the $C^{*}$-correspondence $\mathcal{X}$, satisfies the additional
two hypotheses:
\begin{enumerate}
\item The Fourier coefficients of $m_{1}$are $O((1+k^{2})^{-1}).$
\item $\vert m_{1}(1)\vert=1$
\item For all $x\in[-\frac{\pi}{2},\frac{\pi}{2}]$, $m_{1}(e^{ix})\neq0$.
\end{enumerate}
Then the product $\prod_{k=1}^{\infty}m_{1}(e^{2\pi i2^{-k}t})$ converges
on $\mathbb{R}$ and the limit, $\phi$, lies in $L^{2}(\mathbb{R})$.
Further, for all $x\in\mathbb{R}$

\begin{enumerate}
\item $\phi(2x)=m_{1}(e^{2\pi ix})\phi(x)$ 
\item $\sum_{k\in\mathbb{Z}}\vert\phi(x+k)\vert^{2}=1$. 
\end{enumerate}
\end{thm}

\begin{rem}
Mallat's hypotheses are labeled as equations (38)--(41) on page 76
of \cite{sM89}. Equation (38) is our hypothesis 1, equation (39)
is our hypothesis 2., and his equation (41) is our hypothesis 3. Equation
(40) is the assertion that $m_{1}$ is a unit vector in $\mathcal{X}$.
We note, however, that there is a lot of {}``wiggle room'' in these
hypotheses and a lot of work has been devoted to finding their exact
limits. In \cite[Chapter 6]{iD92}, for example, Daubechies discusses
aspects of this matter at length and exposes, in particular, works
of Cohen and Lawton which give necessary and sufficient conditions
for a unit vector $m_{1}$ to be a \emph{trigonometric polynomial}
and generate a wavelet. The point to keep in mind, for our purposes,
is that a unit vector $m_{1}\in\mathcal{X}$ always generates an isometry $S_1$.
Further, the minimal unitary extension of $\pi (S_1)$, $U$, lives on the space
$L^{2}(\mathbb{T}_{\infty},\tilde{\mu})$, where $\tilde{\mu}$ is constructed
using $m_1$. These things do not depend on anything other than the fact that
$m_1$ is a unit vector in $\mathcal{X}$.
However, \emph{some} hypotheses on $m_{1}$ seem to be necessary to get
from $L^{2}(\mathbb{T}_{\infty},\tilde{\mu})$ to $L^{2}(\mathbb{R})$.
Conclusion 1. of Mallat's theorem is the stepping stone that takes us
from $L^{2}(\mathbb{T}_{\infty},\tilde{\mu})$ to $L^{2}(\mathbb{R})$.
Conclusion 2. does not play an immediate role in the Larsen-Raeburn
approach, but it implies, in particular, that translates of $\hat{\phi}$
are orthonormal \cite[Equation (50)]{sM89}. 
\end{rem}

With the aid of $\phi$ we may define $R_{n}:H_{n}\rightarrow L^{2}(\mathbb{R})$
via the formula $(R_{n}\xi)(x):=2^{\frac{-n}{2}}\xi(e^{2\pi
i(2^{-n}x)})\phi(2^{-n}x)$,
$\xi\in H_{n}(=L^{2}(\mathbb{T}))$. It then is a simple matter to
check that $R_{n}$ is an isometry that satisfies the equation
$R_{n+1}\pi(S_{1})=R_{n}$.
By properties of inductive limits, we may conclude that there is a
unique Hilbert space isometric injection $R_{\infty}:H_{\infty}\rightarrow
L^{2}(\mathbb{R})$
so that $R_{\infty}S_{\infty,n}=R_{n}$. The problem now is to show
that $R_{\infty}$ is surjective. For this purpose, define
$\mathcal{V}_{n}:=R_{n}H_{n}=R_{n}L^{2}(\mathbb{T}).$
Then, as Mallat showed in the proof of the second half of Theorem
2 in \cite{sM89}, we have:

\begin{lem}
\label{Mallat}(Mallat) 
\begin{enumerate}
\item $\mathcal{V}_{n}\subseteq\mathcal{V}_{n+1}$. 
\item $\cap{\mathcal{V}_{n}}=\{0\}$. 
\item $\bigvee\mathcal{V}_{n}=L^{2}(\mathbb{R})$. 
\end{enumerate}
Thus $R_{\infty}:H_{\infty}\rightarrow L^{2}(\mathbb{R})$ is a Hilbert
space isomorphism. 
\end{lem}
We want to remark that conditions 2. and 3. of the preceding lemma
represent two different problems. Condition 2. is the assertion that
the isometry $\pi(S_{1})$ is a pure isometry. Bratteli and Jorgensen
provide a proof of this that is different from Mallat's by noting
that $\pi(S_{1})$ is pure because $m_{1}$ does not have modulus
one a.e. \cite[Theorem 3.1]{BJ97}. The condition 3. also has alternate
proofs. One that we find particularly attractive, because it works
in the more general setting of wavelets built on $\mathbb{T}^{n}$
using a dilation matrix $A$, may be found in Strichartz's survey
\cite[Lemma 3.1]{rS94}.

Recall that the ``dilation by $2$ operator'', $D$, is defined on
$L^{2}(\mathbb{R})$ by the formula $D\xi(x)=2^{1/2}\xi(2x)$ and
that $D$ a unitary operator on $L^{2}(\mathbb{R})$.

\begin{lem}
$R_{\infty}UR_{\infty}^{-1}=D,$ and
$R_{\infty}\rho(f)R_{\infty}^{-1}\xi(t)=f(e^{2\pi it})\xi(t)$
for all $f\in C(\mathbb{T}),$ all $\xi$ in $L^{2}(\mathbb{R})$
and all $t\in\mathbb{R}$. 
\end{lem}
\begin{proof}
The first assertion requires only a simple calculation:\begin{multline*}
D(R_{\infty}(S_{\infty,n+1}\xi))(x)=D(R_{n+1}\xi)(x)=2^{1/2}(R_{n+1}\xi)(2x)\\
=2^{1/2}2^{\frac{-(n+1)}{2}}\xi(e^{2\pi
i(2^{-(n+1)}2x)})\phi(2^{-(n+1)}2x)=(R_{n}\xi)(x)\\
=(R_{\infty}S_{\infty,n}\xi)(x)=(R_{\infty}US_{\infty,n+1}\xi)(x).
\end{multline*}
 The second assertion is verified similarly.
\end{proof}
\noindent If we set $\mathcal{W}_{n}=\mathcal{V}_{n+1}\ominus\mathcal{V}_{n}$,
then by Lemma \ref{Mallat}:\[
\bigoplus_{n\in\mathbb{Z}}\mathcal{W}_{n}=L^{2}(\mathbb{R}).\]
 But then we find that

\noindent \begin{eqnarray}
\mathcal{W}_{0} & = &
\mathcal{V}_{1}\ominus\mathcal{V}_{0}=R_{1}L^{2}(\mathbb{T})\ominus
R_{0}L^{2}(\mathbb{T})\label{eq:Wzero}\\
 & = & R_{1}L^{2}(\mathbb{T})\ominus R_{1}\pi(S_{1})L^{2}(\mathbb{T})\nonumber
\\
 & = & R_{1}(L^{2}(\mathbb{T})\ominus\pi(S_{1})L^{2}(\mathbb{T}))\nonumber \\
 & = & R_{1}\pi(S_{2})L^{2}(\mathbb{T}),\nonumber \end{eqnarray}
by the Cuntz relations. So if we set $e_{k}(z)=z^{k}$ and set
$\zeta_{k}:=R_{1}\pi(S_{2})e_{-k}$.
Then $\{\zeta_{k}\}$ is an orthonormal basis for $\mathcal{W}_{0}$
and \begin{eqnarray*}
\zeta_{k}(x)= & 2^{-1/2}(2^{1/2}m_{2}(e^{(2\pi i2^{-1}x)})e^{-2\pi
ikx})\phi(2^{-1}x)\\
= & e^{-2\pi ikx}\zeta(x)\end{eqnarray*}
where \[
\zeta(x):=m_{2}(e^{\pi ix})\phi(2^{-1}x).\]

\noindent Thus, $\{\zeta_{j,k}\}_{j,k=-\infty}^{\infty}$ is an orthonormal
basis for $L^{2}(\mathbb{R})$, where

\begin{eqnarray*}
\zeta_{j,k}(x) & = & D^{j}\zeta_{k}(x)\\
 & = & 2^{j/2}e^{(-2\pi ik2^{j}x)}\zeta(2^{j}x).\end{eqnarray*}

\noindent Consequently, if $\psi$ is the inverse Fourier transform
of $\zeta$, then $\psi$ is a wavelet.

This completes the proof of the following theorem as formulated by
Bratteli and Jorgensen \cite{BJ97} and Larsen and Raeburn \cite{LR06}. 

\begin{thm}
\label{JLR}(Bratteli-Jorgensen, Larsen-Raeburn)

The inverse Fourier transform of\[
m_{2}(e^{\pi ix})\phi(2^{-1}x)\]
 is the wavelet associated with the filter bank $(m_{1},m_{2})$. 
\end{thm}

\section{Further Thoughts: Fractafolds}

As we have seen, the $C^{*}$-algebra $C^{*}(G)$ always contains
an isometry $S$ and a Cuntz family of isometries $\{ S_{i}\}_{i=1}^{n}$,
provided $\mathcal{X}$ has an orthonormal basis. Further, we may
construct the minimal unitary extension of either $S$ or of any of
the $S_{i}$ \emph{essentially within} $C^{*}(G)$. More accurately,
these objects are constructed in the multiplier algebra of the $C^{*}$-algebra
of a Morita equivalent groupoid that we denote by $G_{\infty}$. To
construct $G_{\infty},$ we form an analogue of the $2$-adic solenoid,
viz., the projective limit space
$X_{\infty}:=\{\underline{x}:=(x_{1},x_{2},\ldots)\,:\, T(x_{k+1})=x_{k}\}$,
and we set \[
G_{\infty}=\{(\underline{x},n-m,\underline{y})\,:\, T^{n}x_{1}=T^{m}y_{1}\}.\]
Then $G_{\infty}$ is a groupoid with unit space $X_{\infty}$ that
is Morita equivalent to $G$ in the sense of \cite{MRW87}. Note,
however, that $G_{\infty}$ is not $r$-discrete. It has many Haar
systems that are not in any evident way equivalent. Each transfer
operator $\mathcal{L}_{D}^{*}$ associated with a continuous function
$D$ as in equation (\ref{eq:Transfer}) determines a natural Haar
system on $G_{\infty}$ that reflects special features of $C^{*}(G)$.
In particular, if $D=\vert m\vert$ where $m$ is a unit vector in
$\mathcal{X}$ (so in particular if $m$ is part of an orthonormal
basis for $\mathcal{X}$), then the minimal unitary extension of
$S_{1}=\iota(m)S$
(in the notation of Theorem \ref{DeaconuIsoThm}) lives in the multiplier
algebra of $C^{*}(G_{\infty})$, when $G_{\infty}$ is endowed with
the Haar system determined by $\vert m\vert$.

We observe in passing that Dutkay and Jorgensen \cite{DJ06} associate
the $C^{*}$-crossed product $C(X_{\infty})\rtimes\mathbb{Z}$ to
the setting we have been discussing (where $\mathbb{Z}$ is viewed
as acting on $X_{\infty}$ through the homeomorphism $T_{\infty}$,
which is defined via the formula
$T_{\infty}(x_{1},x_{2},\ldots)=(Tx_{1},x_{1},x_{2},\ldots)$.)
This crossed product lies in the multiplier algebra of $C^{*}(G_{\infty})$.

Thus, our analysis shows that the study of wavelets can be broken
into two pieces. First, there are the structures that are intrinsic
to the geometric setting of a space $X$ with a local homeomorphism
$T$. These include the groupoid $G$ and its $C^{*}$-algebra, the
pseudogroup $\mathfrak{G}$, and the Deaconu correspondence $\mathcal{X}$.
These are the source of isometries and the Cuntz relations - assuming
$\mathcal{X}$ has an orthonormal basis. Each choice of orthonormal
basis gives Cuntz isometries in $C^{*}(G)$ that satisfy equation
(\ref{eq:inner}). Even if $\mathcal{X}$ fails to have an orthonormal
basis, $\mathcal{X}$ will always contain a (\emph{normalized tight})
\emph{frame} in the sense of Frank and Larson \cite[Definition 3.1]{FL99}
(also called a \emph{quasi-basis} in the sense of Watatani \cite{yW90}).
This is a collection of vectors $\{\psi_{i}\}_{i=1}^{n}$ such that
for every $\xi\in\mathcal{X},$
$\xi=\sum_{i=1}^{n}\psi_{i}\langle\psi_{i},\xi\rangle$and
$\langle\xi,\xi\rangle=\sum_{i=1}^{n}\langle\xi,\psi_{i}\rangle\langle\psi_{i},
\xi\rangle$.
Such a collection may be constructed easily with the aid of a partition
of unity subordinate to an open cover of $X$ such that $T$ is a
homeomorphism when restricted to each element of the cover. Much of
the analysis in $C^{*}(G)$ can be accomplished with a frame for $\mathcal{X}$.
The parameters involved in representing the Cuntz relations on Hilbert
space come from the representation theory of $C^{*}(G)$. Even constructing
the minimal unitary extension of $\pi(S_{1})$ involves ingredients
intrinsic to our setting. The groupoid $G_{\infty}$ is Morita equivalent
to $G$ and carries a natural Haar system that may be ``pegged'' to
$S_{1}$ - more accurately, a natural Haar system on $G_{\infty}$
can be constructed from each low pass filter. There will result a
natural multiresolution analysis in $L^{2}(X_{\infty},\tilde{\mu})$. 

To make contact with wavelet basis in $L^{2}(\mathbb{R}^{n})$ for
some $n$, which is the second piece in the study of wavelets, one
must have a mechanism for passing from $L^{2}(X_{\infty},\tilde{\mu})$
to $L^{2}(\mathbb{R}^{n})$. This involves a different set of tools.
In the final analysis, there may not be any naturally constructed
wavelet-like bases in $L^{2}(\mathbb{R}^{n})$ coming from a particular
space and local homeomorphism. One should not despair at this. Rather,
one should focus on building orthonormal bases in $W_{0}$ (the wandering
subspace in equation (\ref{eq:Wzero}) and then push them around to
form an orthonormal basis for all of $L^{2}(X_{\infty},\tilde{\mu})$
using the minimal unitary extension $U$ of $\pi(S_{1})$. After all,
$L^{2}(X_{\infty},\tilde{\mu})$ and the other spaces we have been
discussing are the \emph{naturally} occurring spaces adapted to $X$
and $T$. This effectively is what Dutkay and Jorgensen did in \cite{DJ05a}
and is similar to what Jorgensen and Pedersen did in \cite{JP98}.

We believe the proof of Theorem \ref{JLR} that we presented, which
is due to Larsen and Raeburn \cite{LR06}, can be tweaked to show
a bit more. The $2$-adic solenoid $\mathbb{T}_{\infty}$ is the dual
group of the $2$-adic numbers: the set of all rational numbers whose
denominators are powers of $2$, positive and negative. Since the
$2$-adic numbers form a dense subgroup of $\mathbb{R}$, $\mathbb{T}_{\infty}$
contains a dense copy of $\mathbb{R}$. We believe the measure $\tilde{\mu}$
is supported on this copy of $\mathbb{R}$ and is mutually absolutely
continuous with respect to Lebesgue measure transported there. The
mapping $R_{\infty}$ ought to be, then, just multiplication by (the
square root of) a suitable Radon-Nikodym derivative. 

In another direction, which we find very piquant, we can find Cuntz
families in the $C^{*}$-algebras or their multiplier algebras of
other groupoids that are Morita equivalent to $G$. This raises the
prospect of carrying out groupoid-like harmonic analysis using Cuntz
families of isometries on other spaces that Strichartz has called
\emph{fractafolds} - i.e. spaces that are locally like fractals \cite{rS03}.
The point is that under favorable circumstances $G$ is the groupoid
of germs of the pseudogroup $\mathfrak{G}$ of partial homeomorphisms
defined by $T$. We believe the pseudogroup of partial homeomorphisms
of a fractafold that is locally like $X$ will be Morita equivalent,
in a sense described by Renault in \cite[Section 3]{jR00}, to $\mathfrak{G}$.
This sense is based on work of Kumjian \cite{aK84} and Haefliger
\cite{aH85}. At this stage, however, there still is a lot of work
to do to substantiate this belief.


\begin{thebibliography}{10}
\bibitem{mB88} M. Barnsely, \emph{Fractals Everywhere}, Academic
Press, Inc., San Diego, 1988.

\bibitem{BJ97}O. Bratteli and P. Jorgensen, \emph{Isometries, Shifts,
Cuntz algebras, and multiresolution wavelet analysis of scale $N$},
Integr. Equ. Oper. Theory \textbf{28} (1997), 382--443. 

\bibitem{BJ02}O. Bratteli and P. Jorgensen, \emph{Wavelets Through
a Looking Glass}, Birkh\"{a}user, Boston, Basel, Berlin, 2002.

\bibitem{BR06}N. Brownlowe and I. Raeburn, \emph{Exel's crossed product
and relative Cuntz-Pimsner algebras}, Math Proc. Camb. Phil. Soc.
\textbf{141} (2006), 497--508.

\bibitem{iD92}I. Daubechies, \emph{Ten Lectures on Wavelets}, CBMS-NSF
Regional Conference Series in Applied Mathematics \textbf{61}, SIAM,
Philadelphia, 1992.


\bibitem{vD95}V. Deaconu, \emph{Groupoids associated with endomorphisms},
Trans. Amer. Math. Soc. \textbf{347} (1995), 1779--1786.

\bibitem{vD99}V. Deaconu, \emph{Generalized solenoids and $C^{*}$-algebras},
Pacific J. Math. \textbf{190} (1999), 247--260.

\bibitem{DKM01}V. Deaconu, A. Kumjian, and P. Muhly, \emph{Cohomology
of topological graphs and Cuntz-Pimsner algebras}, J. Operator Theory
\textbf{46} (2001), 251--264.

\bibitem{rD69}R. G. Douglas, \emph{On extending commutative semigroups
of isometries}, Bull. London Math. Soc. \textbf{1} (1969), 157--159.

\bibitem{DJ05}D. Dutkay and P. E. T. Jorgensen, \emph{Wavelet constructions
in non-linear dynamics}, Electron. Res. Announc. Amer. Math. Soc.
\textbf{11} (2005), 21--33.

\bibitem{DJ05a}D. Dutkay and P. E. T. Jorgensen, Wavelets on fractals,
Rev. Mat. Iberoamericana 22 (2005), 131--180.

\bibitem{DJ05b}D. Dutkay and P. E. T. Jorgensen, Hilbert spaces of
martingales supporting certain substitution-dynamical systems, Conform.
Geom. Dyn. 9 (2005), 24--45.

\bibitem{DJ06}D. Dutkay and P. E. T. Jorgensen, Hilbert spaces built
on a similarity and on dynamical renormalization, J. Math. Phys. 47
(2006), 053504, 20pp.

\bibitem{DJ06a}D. Dutkay and P. E. T. Jorgensen, Iterated function
systems, Ruelle operators, and invariant projective measures, Math.
Comp. 75 (2006), 1931--1970.

\bibitem{DR06}D. Dutkay and K. Roysland, \emph{The algebra of harmonic
functions for a matrix-valued transfer operator}, preprint FA/0611539.

\bibitem{DR07}D. Dutkay and K. Roysland, \emph{Covariant representations
for matrix-valued transfer operators}, FA/0701453.

\bibitem{rE03}R. Exel, \emph{A new look at the crossed-product of
a $C^{*}$-algebra by an endomorphism}, Ergod. Th. \& Dynam. Sys.
\textbf{23} (2003), 1733--1750.

\bibitem{FMR03}N. Fowler, P. Muhly and I. Raeburn, \emph{Representations
of Cuntz-Pimsner algebras}, Indiana U. Math. J. \textbf{52} (2003),
569--605.

\bibitem{FL99}M. Frank and D. Larson, \emph{A module frame concept
for Hilbert $C^{*}$-modules}, 207--234, in \emph{The Functional and
Harmonic Analysis of Wavelets and Frames}, L. Baggett and D. Larson,
Eds., Contemporary Math. \textbf{247}, Amer. Math. Soc., Providence,
1999.

\bibitem{aH85}A. Haefliger, \emph{Pseudogroups of local isometries},
Pitman Res. Notes in Math. \textbf{131}, Longman, Harlow 1985, 174--197.

\bibitem{pJ06}P. E. T. Jorgensen, \emph{Analysis and Probability},
Graduate Texts in Mathematics, \textbf{234}, Springer, New York, 2006.

\bibitem{JP98}P. E. T. Jorgensen and Steen Pedersen, \emph{Dense
analytic subspaces in fractal $L^{2}$-spaces}, J. Anal. Math. \textbf{75}
(1998), 185 -- 228.

\bibitem{aK84}A. Kumjian, \emph{On localizations and simple $C^{*}$-algebras},
Pacific J. Math \textbf{112} (1984), 141--192.

\bibitem{aK85}A. Kumjian, \emph{Diagonals in algebras of continuous
trace}, in \emph{Operator Algebras and their Connections with Topology
and Ergodic Theory}, Lecture Notes in Mathematics Vol. \textbf{1132}
, Springer-Verlag, New York, 1985.

\bibitem{LR06}N. Larsen and I. Raeburn, \emph{From filters to wavelets
via direct limits}, in \emph{Operator theory, operator algebras, and
applications}, 35--40, Contemp. Math. \textbf{414}, Amer. Math. Soc.,
Providence, RI, 2006.

\bibitem{sM89}S. Mallat, \emph{Multiresolution approximations and
wavelet orthonormal bases in} $L^{2}(\mathbb{R})$, Trans. Amer. Math.
Soc. 315 (1989), 69--87.

\bibitem{MRW87}P. Muhly, J. N. Renault and D. Williams, \emph{Equivalence
and isomorphism of groupoid $C^{*}$-algebras}, J. Operator Theory
\textbf{17} (1987), 3--22.

\bibitem{MS98}P. Muhly and B. Solel, \emph{Tensor algebras over
$C^{*}$-correspondences:
representations, dilations, and $C^{*}$-envelopes}, J. Funct. Anal.
\textbf{158} (1998), 389--457.

\bibitem{PR03}J. Packer and M. Rieffel, \emph{Wavelet filter functions,
the matrix completion problem, and projective modules over} $C(\mathbb{T}^{n})$,
J. Fourier Anal. Appl. \textbf{9} (2003), 101--116.

\bibitem{PR04}J. Packer and M. Rieffel, \emph{Projective multi-resolution
analyses for} $L^{2}(\mathbb{R}^{2})$, J. Fourier Anal. Appl. \textbf{10}
(2004), 439--464.

\bibitem{mP97}M. Pimsner, \emph{A class of $C^{*}$-algebras generalizing
both Cuntz-Krieger algebras and crossed products by $\mathbb{Z}$},
Fields Inst. Commun. \textbf{12} (1997), 189--212.

\bibitem{jR80}J. N. Renault, \emph{A Groupoid Approach to $C^{*}$-algebras},
Lecture Notes in Mathematics, Vol. \textbf{793}, Springer-Verlag,
New York, 1980.

\bibitem{jR87}J.N. Renault, \emph{Repr\'esentation des produits
crois\'es d'alg\`ebres de groupo\"ides}, J. Operator Theory \textbf{18}
(1987), 67--97.

\bibitem{jR00}J. N. Renault, \emph{Cuntz-like algebras}, in \emph{Operator
theoretical methods (Timi\c soara, 1998)}, 371--386, Theta Found.,
Bucharest, 2000.

\bibitem{jR03}J. N. Renault, \emph{AF equivalence relations and their
cocycles}, in \emph{Operator Algebras and Mathematical Physics
(Constan\textbackslash{}c
ta, 2001)}, 365--377, 2003.

\bibitem{jR05}J. N. Renault, \emph{The Radon Nikodym property for
approximately proper equivalence relations}, Ergodic Theory Dynam.
Systems, \textbf{25} (2005), 1643--1672.

\bibitem{pS93}P. Stacey, \emph{Crossed products of $C^{*}$-algebras
by $*$-endomorphisms}, J. Austral. Math. Soc. \textbf{54} (1993),
204--212.

\bibitem{rS94}R. Strichartz, \emph{Construction of orthonormal wavelets}
in \emph{Wavelets: mathematics and applications}, 23--50, Stud. Adv.
Math., CRC, Boca Raton, FL, 1994.

\bibitem{rS03}R. Strichartz, \emph{Fractafolds}, Trans. Amer. Math.
Soc. \textbf{355} (2003), 4019--4043.

\bibitem{yW90}Y. Watatani, \emph{Index for $C^{*}$-subalgebras},
Memoires Amer. Math. Soc. \textbf{424} (1990), 83pp.
\end{thebibliography}
\end{document}